\documentclass{amsart}%
\usepackage{amsmath}
\usepackage{amsfonts}
\usepackage{amssymb}
\usepackage{graphicx}%
\providecommand{\U}[1]{\protect\rule{.1in}{.1in}}
\newtheorem{theorem}{Theorem}
\theoremstyle{plain}

\newtheorem{example}{Example}

\newtheorem{proposition}{Proposition}
\newtheorem{remark}{Remark}

\numberwithin{equation}{section}
\begin{document}
\title[Kesten moments]{Yet another way of calculating moments of the Kesten's distribution and its
consequences for Catalan numbers and Catalan triangles}
\author{Pawe\l \ J. Szab\l owski}
\address{Emeritus in Department of Mathematics and Information Sciences, \\
Warsaw University of Technology\\
ul Koszykowa 75, 00-662 Warsaw, Poland}
\thanks{The author would like to thank the unknown referee for his long and detailed
reports containing, typos, errors, and most importantly a list of suggestions,
that substantially helped to improve the paper. I'm impressed and very
grateful for it.}
\date{June, 2021}

\begin{abstract}
We calculate moments of the so-called Kesten distribution by means of the
expansion of the denominator of the density of this distribution and then
integrate all summands with respect to the semicircle distribution. By
comparing this expression with the formulae for the moments of Kesten's
distribution obtained by other means, we find identities involving polynomials
whose power coefficients are closely related to Catalan numbers, Catalan
triangles, binomial coefficients. Finally, as applications of these identities
we obtain various interesting relations between the aforementioned numbers,
also concerning Lucas, Fibonacci and Fine numbers.

\end{abstract}
\maketitle

\section{Introduction}

The purpose of this note is to calculate a sequence of moments of the Kesten's
distribution and thus, by comparison with the existing formulae, to obtain
some polynomial type identities involving Catalan and some other sequences of
numbers related to them (see Proposition \ref{main}). In 2015 in \cite{Szab15}
and in 2020 in \cite{Szab-Kes} Szab\l owski calculated in two different ways
the moments of Kesten distribution. Later T. Hasegawa and S. Saito in
\cite{HaSe21} evaluated these moments in some other ways and, equating the
results, they found interesting identities involving Catalan and related
numbers. So in this note, we will calculate these moments in yet another way
and obtain some other identities, involving, surprisingly, other important
numbers sequences like Fibonacci, Lucas and Fine numbers.

\section{Basic Ingredients}

We start with the modified semicircle distribution, i.e., distribution with
the density
\[
f_{S}(x)=\left\{
\begin{array}
[c]{ccc}%
\frac{1}{2\pi}\sqrt{4-x^{2}} & if & \left\vert x\right\vert \leq2\\
0 & if & \left\vert x\right\vert >2
\end{array}
\right.  .
\]
It is well known that the famous Catalan numbers are the moment sequence of
$f_{S}$. More precisely, we have%
\begin{equation}
\int_{-2}^{2}x^{2n}f_{S}(x)dx=C_{n}=\frac{1}{n+1}\binom{2n}{n}. \label{mp=r}%
\end{equation}
We also know (see, e. g.,\cite{Hor07} (4.8 p. 107)) that, after inserting
proper values of parameters, the moment generating function of this
distribution is equal to%
\[
g_{S}(z)=\frac{2}{\sqrt{1-4z^{2}}+1},
\]
for $\left\vert z\right\vert \leq1/2.$ One can also easily notice, that
\begin{equation}
\sum_{k\geq0}t^{k}C_{k}\allowbreak=\allowbreak\frac{2}{\sqrt{1-4t}+1},
\label{genC}%
\end{equation}
which is valid for $\left\vert t\right\vert \leq1/4$, as the evaluation of the
convergence radius and the study of the convergence interval show, thanks to
the relation:
\begin{equation}
\binom{2n}{n}/4^{n}\allowbreak\underset{n\rightarrow+\infty}{=}\allowbreak
O(n^{-1/2}). \label{ON}%
\end{equation}

The other ingredient is the definition of the Kesten distribution. It was
considered in many papers including \cite{Kesten59}, \cite{Szab15},
\cite{Szab-Kes}, and recently in \cite{HaSe21} with different parametrization.
Let us consider the Kesten distribution parametrized basically as in
\cite{HaSe21} with parameter $q$ replaced by $r$, i. e., with the following
density:%
\begin{equation}
f_{K}(x|p,r)=\left\{
\begin{array}
[c]{ccc}%
0 & if & \left\vert x\right\vert >2\sqrt{r}\\
\frac{p}{2\pi}\frac{\sqrt{4r-x^{2}}}{p^{2}-(p-r)x^{2}} & if & \left\vert
x\right\vert \leq2\sqrt{r}%
\end{array}
\right.  , \label{Kes}%
\end{equation}
for $0<p\leq2r.$

We point out that, if $p=r$, then
\begin{equation}
f_{K}(x|r,r)\allowbreak=\allowbreak\frac{1}{\sqrt{r}}f_{S}(x/\sqrt{r}).
\label{p=r}%
\end{equation}

\section{Main results}

Thus, we have
\[
\int_{-2\sqrt{r}}^{2\sqrt{r}}x^{2n}f_{K}(x|r,r)dx=r^{n}C_{n}=\frac{r^{n}}%
{n+1}\binom{2n}{n}.
\]
Now notice that for
\begin{equation}
\left\vert \frac{p-r}{p^{2}}\right\vert x^{2}\leq\left\vert \frac{p-r}{p^{2}%
}\right\vert 4r<1, \label{cond}%
\end{equation}
we have the following expansion:%
\[
f_{K}(x|p,r)=\frac{r}{p}\frac{\sqrt{4r-x^{2}}}{2\pi r}\sum_{k\geq0}\left(
\frac{p-r}{p^{2}}\right)  ^{k}x^{2k}.
\]
When we consider $\left\vert x\right\vert \leq2\sqrt{r}$ and parameters
satisfying (\ref{cond}), the series for $f_{K}(x|p,r)$ is uniformly convergent
and we can integrate term by term obtaining%
\begin{equation}
M_{2m}(p,r)=\int_{-2\sqrt{r}}^{2\sqrt{r}}x^{2m}f_{K}(x|p,r)dx=\frac{r}{p}%
\sum_{k\geq0}\left(  \frac{p-r}{p^{2}}\right)  ^{k}r^{k+m}C_{k+m}. \label{mom}%
\end{equation}
Let us introduce a new auxiliary variable
\begin{equation}
t\allowbreak=\allowbreak\frac{r}{p}. \label{z}%
\end{equation}
General conditions on $p$ and $r$ require that $t\geq1/2$. From inequalities
(\ref{cond}), we must also have $\left\vert t(1-t)\right\vert <1/4$, however
taking into account (\ref{ON}), we can notice that (\ref{mom}) converges also
for $\left\vert t(1-t)\right\vert \allowbreak=\allowbreak1/4.$ This leads to
the following condition
\begin{equation}
1/2\leq t\leq(1+\sqrt{2})/2. \label{ogr}%
\end{equation}

Summarizing, we get the following result.

\begin{theorem}
Let the Kesten distribution be defined by the density $f_{K}$ given by
(\ref{Kes}). Let the real positive parameters $p$ and $r$ satisfy the
following relationship%
\[
2r\geq p\geq2(\sqrt{2}-1)r,
\]
then we have%
\begin{equation}
M_{2m}(p,r)=\frac{p^{m}}{(1-t)^{m}}\left(  1-\sum_{k=0}^{m-1}t^{k+1}\left(
1-t\right)  ^{k}C_{k}\right)  , \label{nm}%
\end{equation}
with $t$ given by (\ref{z}), with an obvious condition $\left\vert
t(1-t)\right\vert \leq1/4$ and $t\neq1$. For $t\allowbreak=\allowbreak1$ we
have $M_{2m}(r,r)\allowbreak=\allowbreak r^{m}C_{m}$, because of (\ref{mp=r})
and (\ref{p=r}).
\end{theorem}

\begin{proof}
Keeping in mind that $r/p\allowbreak=t$, let us analyze first when the series
(\ref{mom}) is convergent. Namely, it is absolutely convergent if
\[
\left\vert t(1-t)\right\vert \leq\frac{1}{4},
\]
because of (\ref{ON}), that is, when%
\[
-\frac{1}{4}\leq t-t^{2}\leq\frac{1}{4}.
\]
The inequality $\frac{r}{p}-(\frac{r}{p})^{2}\leq\frac{1}{4}$ is equivalent to
the following $(\frac{1}{2}-\frac{r}{p})^{2}\geq0$ which is always true when
$p\neq0,$ in particular for all $p$ and $r$ such that $2r\geq p>0.$ The second
one leads to inequality
\[
t^{2}-t-\frac{1}{4}\leq0,
\]
which gives $\frac{1-\sqrt{2}}{2}\leq t\leq\frac{1+\sqrt{2}}{2}$, a condition
that is always satisfied when $2r\geq p\geq2(\sqrt{2}-1)r>0$.

From the ordinary generating function (\ref{genC}) of Catalan numbers, since
$\left\vert t(1-t)\right\vert \leq1/4$, we also have the following identity:
\[
\sum_{n\geq0}t^{n}(1-t)^{n}C_{n}\allowbreak=\allowbreak\frac{1}{t},
\]
Using this identity, since $\left\vert t(1-t)\right\vert \leq1/4$, from
(\ref{mom}) when $t\neq1$ we find
\begin{gather*}
M_{2m}(p,r)=\frac{r}{p}\frac{p^{2m}}{(p-r)^{m}}\sum_{k\geq0}t^{k+m}%
(1-t)^{k+m}C_{k+m}=\\
=\frac{r}{p}\frac{p^{m}}{(1-\frac{r}{p})^{m}}\left(  \sum_{k\geq0}%
t^{k}(1-t)^{k}C_{k}-\sum_{k=0}^{m-1}t^{k}(1-t)^{k}C_{k}\right)  =\\
=p^{m}\frac{t}{(1-t)^{m}}\left(  \frac{1}{t}-\sum_{k=0}^{m-1}t^{k}%
(1-t)^{k}C_{k}\right)  .
\end{gather*}
Now it suffices to multiply the expression in round brackets by $t$.
\end{proof}

\begin{remark}
\label{refer}Formula (\ref{nm}) can be derived from the unnumbered formula%
\[
M_{2k}=\frac{p}{p-q}\left(  \frac{p^{2k-1}}{(p-q)^{k-1}}-\sum_{m=1}^{k}%
\frac{\binom{2m-1}{m}}{m}q^{m}\frac{p^{2(k-m)}}{(p-q)^{k-m}}\right)
\]
placed in Comment 1 of \cite{HaSe21} thanks to equality%
\[
C_{m-1}=\frac{1}{m}\binom{2m-2}{m-1}=\frac{1}{2m-1}\binom{2m-1}{m}%
\]
and keeping in mind that in our notation $q=r,$ $p\neq r$ and $t\allowbreak
=\allowbreak r/p$. Indeed we have
\begin{align*}
M_{2k}  &  =\frac{p}{p-r}\left(  \frac{p^{2k-1}}{(p-r)^{k-1}}-\sum_{m=1}%
^{k}\frac{\binom{2m-1}{m}}{m}r^{m}\frac{p^{2(k-m)}}{(p-r)^{k-m}}\right)  =\\
&  =\frac{p}{p-r}\left(  \frac{p^{2k-1}}{(p-r)^{k-1}}-\sum_{j=0}^{k-1}%
C_{j}r^{j+1}\frac{p^{2(k-j-1)}}{(p-r)^{k-j-1}}\right)  =\\
&  =\frac{p^{2k}}{\left(  p-r\right)  ^{k}}\left(  1-\sum_{j=0}^{k-1}%
C_{j}r^{j+1}\frac{1}{p}\left(  \frac{p-r}{p^{2}}\right)  ^{j}\right)  .
\end{align*}
This calculation was done by the referee in his report.

Let us underline the important property of the sequences that we are
considering, namely, that the sequences :
\[
\left\{  M_{2m}\right\}  _{m\geq0},\left\{  \frac{1}{(1-t)^{m}}\left(
1-\sum_{k=0}^{m-1}t^{k+1}\left(  1-t\right)  ^{k}C_{k}\right)  \right\}
_{m\geq0}%
\]
and also the sequences for $d\in\lbrack0,1]$
\[
\left\{  1-d\sum_{k=0}^{m-1}t^{k+1}\left(  1-t\right)  ^{k}C_{k}\right\}
_{m\geq0},
\]
are moment sequences. The last statements follow from the fact that the
product of two moment sequences is another moment sequence and that the convex
combination of two moment sequences is another moment sequence. For a brief
recollection of facts concerning the moment sequences, see e.
g.,\cite{Sokal20} or Appendix in \cite{Szabl21}.
\end{remark}

\begin{remark}
As a matter of honesty, the author was able to see the first version of the
paper of Hasegawa and Saito. In this version, there were not present Comment 1
and Comment 2. It was in May and June 2021. Their paper has inspired the
author to write this note. To clarify everything, the final form of the paper
\cite{HaSe21} appeared at the end of October 2021. In the final form, the
authors included three Comments. In the first Comment the formula from the
first line of the Remark \ref{refer} appeared while in Comment 2 the expansion
(\ref{mom}) appeared. In the third comment, the authors stated that these
formulae are promising and that they will research further on these formulae.
Anyway, the Remark \ref{refer} indicates that to get the crucial formula
(\ref{nm}) one did not need to exploit the new way of calculating even moments
of Kesten distribution.
\end{remark}

Let us now compare this result with known formulae for the moments of Kesten distribution.

Let us notice that we have the following equality%
\[
f_{CN}(x|0,\rho,0)=f_{K}(x\sqrt{r},p,r)\sqrt{r}%
\]
involving the distribution $f_{CN}(x|y,\rho,q)$ considered in \cite{Szab15}
with $y\allowbreak=\allowbreak q\allowbreak=\allowbreak0$ and $\rho
^{2}\allowbreak=\allowbreak1-p/r$ and the distribution defined in (\ref{Kes}).
Thanks to this equality and to Proposition 3 part (i) in \cite{Szab15}, we
find%
\[
M_{2m}(p,r)=r^{m}\sum_{k=0}^{m}\left(  \frac{p}{r}-1\right)  ^{m-k}S_{m,k},
\]
where
\[
S_{m,k}\allowbreak=\allowbreak\binom{2m}{k}-\binom{2m}{k-1},
\]
with the understanding that $\binom{2m}{-1}\allowbreak=\allowbreak0.\ $Hence
in terms of $t\allowbreak=\allowbreak r/p$ and $t\neq1$, we have:%
\begin{equation}
M_{2m}(p,r)=p^{m}\sum_{k=0}^{m}\left(  -\frac{r}{p}+1\right)  ^{m-k}\left(
\frac{r}{p}\right)  ^{k}S_{m,k}=p^{m}\sum_{k=0}^{m}t^{k}\left(  1-t\right)
^{m-k}S_{m,k}. \label{mS}%
\end{equation}
In \cite{HaSe21} two other expressions for the moments of Kesten distribution
can be found. Namely the following formulae:%
\begin{align}
M_{2m}(p,r)  &  =p\sum_{j=0}^{m-1}p^{m-1-j}r^{j}T_{m-1,j},\label{m1}\\
M_{2m}(p,r)  &  =p\sum_{j=0}^{m-1}(p-r)^{j}r^{m-1-j}B_{m,j+1}, \label{m2}%
\end{align}
where numbers $T_{m,j}$ and $B_{m,j}$ are called Catalan triangles, depending
on the author. The numbers $T_{m,j}$ are defined as
\[
T_{m,j}=\frac{m-j+1}{m+1}\binom{m+j}{m}%
\]
for integers $m,j$ such that $m\geq j\geq0$ and the related sequence is
A009766 in Sloane's Encyclopedia \cite{Sloan}, the numbers $B_{k,j}$,
introduced by L.W. Shapiro in \cite{Shapiro76} and with related sequence
A039598 in OEIS \cite{Sloan}, are given by:%
\[
B_{k,j}=\frac{j}{k}\binom{2k}{k-j},
\]
where the integers $k,j$ satisfy $k\geq j\geq1.$ Comparing formulae (\ref{mS})
(\ref{m1}) and (\ref{m2}) we find the following result.

\begin{proposition}
\label{main}i) For all $m\geq1$ and $t\in\mathbb{C}$ we obtain:
\begin{gather}
\left(  1-t\right)  ^{m}\sum_{k=0}^{m}S_{m,k}t^{k}\left(  1-t\right)
^{m-k}=\left(  1-t\right)  ^{m}\sum_{k=0}^{m-1}T_{m-1,k}t^{k}=\label{=1}\\
=\left(  1-t\right)  ^{m}\sum_{k=0}^{m-1}B_{m,k+1}\left(  1-t\right)
^{k}t^{m-1-k}=1-\sum_{k=0}^{m-1}C_{k}t^{k+1}\left(  1-t\right)  ^{k}.
\label{=2}%
\end{gather}

ii) For all $m\geq1$ and $x\in\mathbb{C}$ we get:
\begin{gather}
\sum_{k=0}^{m}S_{m,k}x^{k}=(x+1)\sum_{k=0}^{m-1}B_{m,k+1}x^{m-1-k}%
=\label{=11}\\
=\sum_{k=0}^{m-1}T_{m-1,k}x^{k}(x+1)^{m-k}=(1+x)^{2m}-\sum_{k=0}^{m-1}%
C_{k}x^{k+1}(x+1)^{2m-2k-1}. \label{=22}%
\end{gather}

\end{proposition}

\begin{proof}
i) Firstly, for $t$ satisfying (\ref{ogr}) we have from (\ref{nm}) and
(\ref{mS}) the following equalities, which are true also when $t=1$:%
\[
\frac{(1-t)^{m}}{p^{m}}M_{2m}=(1-t)^{m}\sum_{k=0}^{m}S_{m,k}t^{k}%
(1-t)^{m-k}=1-\sum_{k=0}^{m-1}t^{k+1}\left(  1-t\right)  ^{k}C_{k}.
\]
Now with parametrization $t=r/p$ formulae (\ref{m1}) and (\ref{m2}) become:%
\begin{gather*}
\frac{(1-t)^{m}}{p^{m}}M_{2m}=(1-t)^{m}\sum_{k=0}^{m-1}T_{m-1,k}t^{k}=\\
=(1-t)^{m}\sum_{k=0}^{m-1}B_{m,k+1}t^{m-1-k}(1-t)^{k}.
\end{gather*}
Therefore, we obtain the chain of equalities given by (\ref{=1}) and
(\ref{=2}). Finally, we observe that all these equalities involve polynomials
in $t$, so we extend their domain from any segment to all complex numbers and
conclude that they hold for all $t\in\mathbb{C}$.

ii) Having proved i) we consider $x\allowbreak=\allowbreak t/(1-t)$, with
$t\neq1$. Then $t\allowbreak=\allowbreak x/(x+1)$ and we consider the
identities (\ref{=1}) and (\ref{=2}) for $x\neq-1.$ Now we multiply both sides
of each of them by $(1+x)^{2m}$. We get immediately forms (\ref{=11}) and
(\ref{=22}). Now again we deal with polynomials hence we can drop assumption
that $x\neq-1$.
\end{proof}

\section{Applications}

The equalities proved in Proposition 1 could have many useful applications.
Thanks to them, we can find relationships between Catalan numbers, Catalan
triangles, binomial coefficients and other special numbers like, e. g.,
Fibonacci and Lucas or Fine numbers. Indeed, in the next two examples we show
some interesting identities obtained respectively from formulae (\ref{=1}),
(\ref{=2}) and formulae (\ref{=11}), (\ref{=22}).

\begin{example}
\label{oldpar} Let us consider formulae (\ref{=1}) and (\ref{=2}), evaluating
them for some special values of $t\in\mathbb{C}$ we find other interesting identities.

a) For $t\allowbreak=\allowbreak2$, $1-t\allowbreak=\allowbreak-1$ and
$t/(1-t)\allowbreak=\allowbreak-2$ and finally we get for all $m\geq1:$%
\begin{gather*}
1-2\sum_{k=0}^{m-1}(-2)^{k}C_{k}=\sum_{k=0}^{m}(-2)^{k}S_{m,k}=\\
=-\sum_{k=0}^{m-1}(-2)^{m-1-k}B_{m,k+1}=(-1)^{m}\sum_{k=0}^{m-1}2^{k}%
T_{m-1,k}.
\end{gather*}
b) For $t\allowbreak=\allowbreak\allowbreak e^{i\pi/3}$, $1-t\allowbreak
=\allowbreak\allowbreak e^{-i\pi/3}$ and $\frac{t}{1-t}\allowbreak
=\allowbreak\allowbreak e^{i2\pi/3}$ and finally we get for all $m\geq1:$%
\begin{gather*}
1-e^{i\pi/3}\sum_{k=0}^{m-1}C_{k}=e^{-im\pi/3}\sum_{k=0}^{m-1}e^{ik\pi
/3}T_{m-1,k}=\\
=e^{-2im\pi/3}\sum_{k=0}^{m}e^{2ik\pi/3}S_{m,k}=e^{-i(2m-1)\pi/3}\sum
_{k=0}^{m-1}e^{i2(m-1-k)\pi/3}B_{m,k+1}.
\end{gather*}
From these formulae considering real parts we find:%
\begin{gather*}
1-\frac{1}{2}\sum_{k=0}^{m-1}C_{k}=\sum_{k=0}^{m-1}\cos((m-k)\pi
/3)T_{m-1,k}=\\
=\sum_{k=0}^{m}\cos(2(m-k)\pi/3)S_{m,k}=\sum_{k=0}^{m-1}\cos((2k+1)\pi
/3)B_{m,k+1},
\end{gather*}
while, when considering imaginary parts, we get:%
\begin{gather*}
\sum_{k=0}^{m-1}C_{k}=\frac{2\sqrt{3}}{3}\sum_{k=0}^{m-1}\sin((m-k)\pi
/3)T_{m-1,k}=\\
=\frac{2\sqrt{3}}{3}\sum_{k=0}^{m}\sin(2(m-k)\pi/3)S_{m,k}=\frac{2\sqrt{3}}%
{3}\sum_{k=0}^{m-1}\sin((2k+1)\pi/3)B_{m,k+1}.
\end{gather*}
We observe that the following functions of $m$ $\cos(m\pi/3)$, $\cos
(2m\pi/3),$ $\frac{2\sqrt{3}}{3}\sin(m\pi/3)$ and $\frac{2\sqrt{3}}{3}%
\sin(2m\pi/3)$ are periodic with periods equal to $6.$ Moreover, we have:
$\cos(m\pi/3)\in\left\{  \pm1/2,\pm1\right\}  $ and $\frac{2\sqrt{3}}{3}%
\sin(m\pi/3)\in\left\{  0,\pm1\right\}  .$ \newline c) For $t\allowbreak
=\allowbreak(1+\sqrt{5})/2$ and $1-t\allowbreak=\allowbreak(1-\sqrt{5})/2$ and
and thanks to the following relations involving Fibonacci numbers $F_{n}$ and
Lucas numbers $L_{n}$ (respectively in OEIS A000045 and A000032):%
\begin{align*}
\left(  \frac{1+\sqrt{5}}{2}\right)  ^{k}  &  =\frac{L_{n}}{2}+\sqrt{5}%
\frac{F_{n}}{2},\\
\left(  \frac{1-\sqrt{5}}{2}\right)  ^{k}  &  =\frac{L_{n}}{2}-\sqrt{5}%
\frac{F_{n}}{2}.
\end{align*}
we get for all $m\geq1:$%
\begin{gather*}
\frac{L_{m}-F_{m}\sqrt{5}}{2}\sum_{k=0}^{m}S_{m,k}\frac{L_{k}+F_{k}\sqrt{5}%
}{2}\frac{L_{m-k}-F_{m-k}\sqrt{5}}{2}=\\
=\frac{L_{m}-F_{m}\sqrt{5}}{2}\sum_{k=0}^{m-1}T_{m-1,k}\frac{L_{k}+F_{k}%
\sqrt{5}}{2}=\\
=\frac{L_{m}-F_{m}\sqrt{5}}{2}\sum_{k=0}^{m-1}B_{m,k+1}\frac{L_{k}-F_{k}%
\sqrt{5}}{2}\frac{L_{m-1-k}+F_{m-1-k}\sqrt{5}}{2}=\\
=1-\frac{L_{1}+F_{1}\sqrt{5}}{2}\sum_{k=0}^{m-1}C_{k}(-1)^{k}.
\end{gather*}
Now we divide all sides by $\frac{L_{m}-F_{m}\sqrt{5}}{2}$ getting
$1/(\frac{L_{m}-F_{m}\sqrt{5}}{2})\allowbreak=\allowbreak(-1)^{m}(\frac
{L_{m}+F_{m}\sqrt{5}}{2})$ and we perform the calculations inside the sums
using also the following identities:
\begin{align*}
L_{n}F_{k}  &  =F_{n+k}+(-1)^{k}F_{k-n},~F_{-k}=(-1)^{k-1}F_{k},\\
L_{n}L_{k}-5F_{n}F_{k}  &  =(-1)^{k}2L_{n-k},~L_{-n}=(-1)^{n}L_{n}.
\end{align*}
Finally, we equate firstly the terms free from the factor $\sqrt{5}$ obtaining%
\begin{gather}
\sum_{k=0}^{m}(-1)^{k}S_{m,k}L_{m-2k}=\sum_{k=0}^{m-1}T_{m-1,k}L_{k}%
=\sum_{k=0}^{m-1}(-1)^{k}B_{m,k+1}L_{m-2k-1}=\label{L2}\\
=(-1)^{m}L_{m}+(-1)^{m+1}L_{m+1}\sum_{k=0}^{m-1}(-1)^{k}C_{k} \label{L3}%
\end{gather}
and then the ones multiplied by $\sqrt{5}$ finding
\begin{gather}
\sum_{k=0}^{m}(-1)^{k}S_{m,k}F_{m-2k}=\sum_{k=0}^{m-1}T_{m-1,k}F_{k}%
=\sum_{k=0}^{m-1}(-1)^{k}B_{m,k+1}F_{m-1-2k}\label{F2}\\
=(-1)^{m}F_{m}-(-1)^{m+1}F_{m+1}\sum_{k=0}^{m-1}(-1)^{k}C_{k}. \label{F3}%
\end{gather}

\end{example}

\begin{example}
\label{newpar} Let us consider identities (\ref{=11}) and (\ref{=22}), we find
some interesting identities using different values of $x$.

a) When $x\allowbreak=\allowbreak1$ we get for $m\geq1:$
\begin{align*}
\sum_{k=0}^{m}S_{m,k}\allowbreak &  =\allowbreak2\sum_{k=0}^{m-1}%
B_{m,k+1}=\sum_{k=0}^{m-1}T_{m-1,k}2^{m-k}=\\
&  =4^{m}-\sum_{k=0}^{m-1}C_{k}2^{2m-2k-1},
\end{align*}

b) When $x\allowbreak=\allowbreak-1/2$ after multiplying all sums by $2^{m} $,
we find for $m\geq1$ $:$%
\begin{gather}
(-1)^{m}\sum_{k=0}^{m}S_{m,k}(-2)^{m-k}=(-1)^{m-1}\sum_{k=0}^{m-1}%
B_{m,k+1}(-2)^{k}=\label{fine1}\\
=\sum_{k=0}^{m-1}(-1)^{k}T_{m-1,k}=\frac{1}{2^{m}}-\sum_{k=0}^{m-1}%
(-1)^{k+1}C_{k}\frac{1}{2^{m-k}}. \label{fine2}%
\end{gather}

c) Let us divide both sides of (\ref{=11}) and (\ref{=22}) by $(x+1)$, if we
pass to the limit $x\rightarrow-1$, taking into account that%
\[
\lim_{x\rightarrow-1}\frac{x^{n}-(-1)^{n}}{x+1}=(-1)^{n-1}n,
\]
and that for $m\geq1$
\[
\sum_{j=0}^{m}S_{m,k}(-1)^{k}=0,
\]
we obtain for all $m\geq1$%
\begin{gather*}
\lim_{x\rightarrow-1}\frac{\sum_{k=0}^{m}S_{m,k}x^{k}}{x+1}=\sum_{k=0}%
^{m}(-1)^{k+1}S_{m,k}k\\
=\sum_{k=0}^{m-1}(-1)^{m-1-k}B_{m,k+1}=(-1)^{m-1}T_{m-1,m-1}=(-1)^{m-1}%
C_{m-1}.
\end{gather*}

\end{example}

\begin{remark}
From Grimaldi's monograph \cite{Grimaldi12}(p. 285) we find, that Fine numbers
$\Phi_{n}$ are given by the relationship:
\[
\Phi_{n}=-\frac{1}{2}\Phi_{n-1}+\frac{1}{2}C_{n},
\]
for $n\geq1$ with $\Phi_{0}\allowbreak=\allowbreak1.$ Using well-known
formulae for the solutions of discrete, non-homogeneous, linear difference
equations we immediately find that
\[
\Phi_{n}=\frac{1}{2}\left(  \frac{-1}{2}\right)  ^{n}+\frac{1}{2}\sum
_{j=0}^{n}C_{j}\left(  \frac{-1}{2}\right)  ^{n-j}.
\]
Now, notice, that the previous formula is in direct relation with formulae
(\ref{fine1}) and(\ref{fine2}). Indeed we have
\[
\frac{1}{2^{m}}-\sum_{k=0}^{m-1}(-1)^{k+1}C_{k}\frac{1}{2^{m-k}}%
=(-1)^{m-1}\Phi_{m-1}.
\]
Hence, dividing every part of the equation (\ref{fine1}) by $(-1)^{m-1}$ and
renaming the indexes, we obtain the following expressions for the Fine
numbers:%
\[
\Phi_{n}=\sum_{j=0}^{n}B_{n+1,j+1}(-2)^{j}=\sum_{j=0}^{n}T_{n,j}%
(-1)^{n-j}=-\sum_{j=0}^{n+1}S_{n+1,j}(-2)^{n+1-j}.
\]

\end{remark}

\end{document}